\documentclass[11pt]{article}
\usepackage{}
\usepackage{enumerate}
\usepackage{amsfonts}
\usepackage{amssymb}
\usepackage{cite}
\usepackage{xcolor}
\usepackage{mathrsfs}

\makeatletter
         \@addtoreset{equation}{section}\makeatother

\newtheorem{theorem}{Theorem}[section]

\newtheorem{lemma}[theorem]{Lemma}

\newtheorem{proposition}[theorem]{Proposition}

\newtheorem{definition}[theorem]{Definition}
\def\bbZ{\mathbb{Z}}

\def\bbR{\mathbb{R}}

\begin{document}
\title{Reconstruction of functions in principal shift-invariant subspaces of mixed
Lebesgue spaces}

\author{Qingyue Zhang\\
\footnotesize\it College of Science,  Tianjin University of Technology,
      Tianjin~300384, China\\
      \footnotesize\it     {e-mails: {jczhangqingyue@163.com}}\\ }

\maketitle

\textbf{Abstract.}\,\,
In this paper, we discuss to the nonuniform sampling problem in principal shift-invariant subspaces of mixed
Lebesgue spaces. We proposed a fast reconstruction algorithm which allows to exactly reconstruct the functions in the principal shift-invariant subspaces as long as the sampling set $X=\{(x_{j},y_{k}):k,j\in \mathbb{J}\}$ is sufficiently dense. Our results improve the result in \cite{LiLiu}.


\textbf{2010 MR Subject Classification}\,\,
94A20, 94A12, 42C15, 41A58

\section{Introduction and motivation}
\ \ \ \
Mixed Lebesgue spaces generalize Lebesgue spaces. It plays a crucial role in the study of time based partial differential equations. In 1961, Benedek firstly proposed  mixed Lebesgue spaces \cite{Benedek,Benedek2}. Mixed Lebesgue spaces was arised due to considering functions that depend on independent quantities with different properties. In the 1980s, Fernandez, and Francia, Ruiz and Torrea developed the theory of mixed Lebesgue spaces in integral operators and Calder\'on--Zygmund operators, respectively \cite{Fernandez,Rubio}. Recently, Torres and Ward, and Li, Liu and Zhang studied sampling problem in shift-invariant subspaces of mixed Lebesgue spaces \cite{Torres,Ward,LiLiu}. In this context, we study the reconstruction of the signals in shift-invariant subspaces of mixed Lebesgue spaces.

Sampling theorem is the theoretical basis of communication system, and is also one of the most powerful basic tools in signal processing. Shannon formally proposed  sampling theorem in 1948 \cite{S,S1}. By Shannon sampling
theorem,  for any $f\in L^{2}(\mathbb{R})$ with $\mathrm{supp}\hat{f}\subseteq[-T,T],$ one has
$$f(x)=\sum_{n\in \mathbb{Z}}f\left(\frac{n}{2T}\right)\frac{\sin\pi(2Tx-n)}{\pi(2Tx-n)}=\sum_{n\in \mathbb{Z}}f\left(\frac{n}{2T}\right)\mathrm{sinc}(2Tx-n),$$
where the series converges uniformly on compact sets and in $L^{2}(\mathbb{R})$, and
$$\hat{f}(\xi)=\int_{\mathbb{R}}f(x)e^{-2\pi i x\xi }dx,\ \ \ \ \ \ \xi\in\mathbb{R}.$$
However, in many realistic situations, Shannon sampling theorem is not applicable, because the known data is only a nonuniform
sampling set. To offset the shortcoming of Shannon sampling theorem, in recent years, there are many results concerning nonuniform sampling problem \cite{SQ,BF1,M,Venkataramani,Christopher,Sun,Zhou}.
Sampling problems
also have been generalized to more general shift-invariant spaces \cite{Aldroubi1,Aldroubi2,Aldroubi3,Aldroubi4,Vaidyanathan,Zhang1} of the form
$$V(\phi)=\left\{\sum_{k\in\mathbb{Z}}c(k)\phi(x-k):\{c(k):k\in\mathbb{Z}\}\in \ell^{2}(\mathbb{Z})\right\}.$$

Torres and Ward firstly studied the sampling problem for band-limited functions in mixed Lebesgue spaces \cite{Torres,Ward}. Then, Li, Liu and Zhang generalized their results to principal shift-invariant spaces of mixed Lebesgue spaces \cite{LiLiu}. However, the generator of shift-invariant spaces studied by Li, Liu and Zhang belongs to Wiener amalgam space $W(L^{1})$. In this paper, we discuss the bigger shift-invariant subspaces generated by a function in  mixed Wiener amalgam space $W(L^{1,1})$ (see Definition \ref{Wieneramalgamspace}). 

The paper is organized as follows. In the next section, we give the definitions and preliminary results needed. Section 3 proposes a fast reconstruction algorithm. Finally, concluding remarks are presented in section 4.

\section{Definitions and preliminary results}
\ \ \ \
In this section, we give some definitions and preliminary results needed in this paper. Firstly, we give the definition of mixed Lebesgue spaces $L^{p,q}(\Bbb R^{d+1})$.
\begin{definition}
For $1 \leq p,q <+\infty$. $L^{p,q}=L^{p,q}(\Bbb R^{d+1})$ consists of all measurable functions $f=f(x,y)$ defined on $\Bbb R\times\Bbb R^{d}$ satisfying
$$\|f\|_{L^{p,q}}=\left[\int_{\Bbb R}\left(\int_{\Bbb R^d}|f(x,y)|^{q}dy\right)^{\frac{p}{q}}dx\right]^{\frac{1}{p}}<+\infty.$$
\end{definition}
The following is the definition of corresponding sequence spaces
 $$\ell^{p,q}=\ell^{p,q}(\mathbb{Z}^{d+1})=\left\{c: \|c\|^{p}_{\ell^{p,q}}=\sum_{k_{1} \in \Bbb Z}\left(\sum_{k_{2} \in \Bbb Z^d }
|c(k_{1},k_{2})|^{q}\right)^{\frac{p}{q}}
<+\infty\right\}.$$

%Obviously, $L^{p,p}(\Bbb R^{d+1})=L^{p}(\Bbb R^{d+1}),\ell^{p,p}(\Bbb Z^{d+1})=\ell^{p}(\Bbb Z^{d+1})=\ell^{p}.$

To controling the local behavior of functions, we introduce mixed Wiener amalgam spaces $W(L^{p,q})(\Bbb R^{d+1})$. 
\begin{definition}\label{Wieneramalgamspace}
For $1\leq p,q<\infty$, if a measurable function $f$ satisfies
$$\|f\|^{p}_{W(L^{p,q})}:=\sum_{n\in \Bbb Z}\sup_{x\in[0,1]}\left[\sum_{l\in \Bbb Z^d}\sup_{y\in [0,1]^d}|f(x+n,y+l)|^{q}\right]^{p/q}<\infty,$$
then we say that $f$ belongs to the mixed Wiener amalgam space $W(L^{p,q})=W(L^{p,q})(\Bbb R^{d+1})$.
\end{definition} 
% For $1\leq p,q<+\infty$,
% $$\|f\|_{W(L^{p}L^{q}(\Bbb R\times\Bbb R))}
% =\left[\sum_{k_{1} \in \Bbb Z}\left(\sum_{k_{2} \in \Bbb Z}
% \sup_{(x,y)\in [0,1]^{2}}|f(x+k_{1},y+k_{2})|^{q}\right)^{\frac{p}{q}}\right]^{\frac{1}{p}}<\infty.$$
%$$\|f\|_{W(L^{\infty}(\Bbb R\times\Bbb R))}=\sup_{(k_{1}, k_{2})
%\in \Bbb Z^{2}}\hbox{esssup}_{(x,y)\in [0,1]^{2}}|f(x+k_{1},y+k_{2})|<\infty.$$

%In order to obtain some basic  properties, we introduce the corresponding conclusion in n-dimension.
For $1\leq p<\infty,$ if a function $f$ satisfies
$$\|f\|^{p}_{W(L^{p})}:=\sum_{k\in \Bbb Z^{d+1}}\mathrm{ess\sup}_{x\in[0,1]^{d+1}} |f(x+k)|^{p}<\infty,$$
then we say that $f$ belongs to the Wiener amalgam space $W(L^{p})=W(L^{p})(\Bbb R^{d+1}).$

For $p=\infty$, if a measurable function $f$ satisfies $$\|f\|_{W(L^{\infty})}:=\sup_{k\in\Bbb Z^{d+1}}\mathrm{ess\sup}_{x\in[0,1]^{d+1}}|f(x+k)|<\infty,$$ then we say that
$f$ belongs to $W(L^{\infty})=W(L^{\infty})(\Bbb R^{d+1})$. Obviously, $W(L^{p})\subset W(L^{p,p}).$

Since sampling theorems only deal with continuous functions, we introduce the definitions of $W_{0}\left ( L^{p,q} \right )$ and $W_{0}\left ( L^{p} \right )$. Let $W_{0}\left ( L^{p,q} \right )\,\left ( 1\leq p,q<\infty  \right )$ denote the space of all continuous functions in $W(L^{p,q})$. Let $W_{0}\left ( L^{p} \right )\,\left ( 1\leq p\leq \infty  \right )$ denote the space of all continuous functions in $W(L^{p})$.

For any sequence $c\in \ell^{p}~(1\leq p\leq +\infty)$ and $\phi\in W(L^{1,1})$, define the semi-discrete convolution of $c$ and $\phi$ by
$$(c*_{sd}\phi)(x)=\sum_{k\in\mathbb{Z}^{d+1}}c(k)\phi(x-k).$$

\subsection{ Principal shift-invariant subspaces in $L^{p,q}$ }

\ \ \ \ For $\phi\in W(L^{1,1})$, the principal shift-invariant space in the mixed Legesgue spaces $L^{p,q}$ is defined by
\begin{eqnarray*}
V_{p,q}(\phi)=\left\{\sum_{k_{1}\in \Bbb Z}\sum_{k_{2}\in \Bbb Z^{d}}c(k_{1},k_{2})\phi(\cdot-k_{1},\cdot-k_{2}):c=\left\{c(k_{1},k_{2}):k_{1}\in \Bbb Z,k_{2}\in\Bbb Z^{d}\right\}\in \ell^{p,q}\right\}.
\end{eqnarray*}
It is easy to see that the double sum pointwisely converges almost everywhere. In fact, since $c=\left\{c(k_{1},k_{2}):k_{1}\in \Bbb Z,k_{2}\in\Bbb Z^{d}\right\}\in\ell^{p,q}$, we have $c\in \ell^{\infty}.$ This combines $\phi\in W(L^{1,1})$ obtians
\begin{eqnarray*}
&&\sum_{k_{1}\in \Bbb Z}\sum_{k_{2}\in \Bbb Z^{d}}\left|c(k_{1},k_{2})\phi(x_{1}-k_{1},x_{2}-k_{2})\right|\\
&& \quad\leq\|c\|_\infty\sum_{k_{1}\in \Bbb Z}\sum_{k_{2}\in \Bbb Z^{d}}|\phi(x_{1}-k_{1},x_{2}-k_{2})|\\
&& \quad\leq\|c\|_\infty\sum_{k_{1}\in \Bbb Z}\sup_{x_{1}\in[0,1]}\sum_{k_{2}\in \bbZ^{d}}\sup_{x_{2}\in[0,1]^{d}}|\phi(x_{1}-k_{1},x_{2}-k_{2})|\\
&&\quad=\|c\|_\infty\| \phi\|_{W(L^{1,1})}<\infty\,(a.e.).
 \end{eqnarray*}

The following proposition gives that the principal shift-invariant spaces are well-defined in $L^{p,q}$.

\begin{proposition}\cite[Theorem 3.1]{LiLiu}\label{pro:stableup1}
Assume that $1\leq p,q<\infty $ and $\phi\in W(L^{1,1})$.
Then for any $c\in \ell^{p,q}$, the function $f=\sum_{k_{1}\in \Bbb Z}\sum_{k_{2}\in \Bbb Z^d}c(k_{1},k_{2})\phi(\cdot-k_{1},\cdot-k_{2})$
 belongs to $L^{p,q}$ and
$$
\|f\|_{L^{p,q}}\leq\|c\|_{\ell^{p,q}}\left \| \phi \right \|_{W(L^{1,1})}.
$$
\end{proposition}

The following proposition gives the norm equivalence. It can be proved by the proof of \cite[Theorem 3.4]{LiLiu}. For convenience, for quantities $X$ and $Y$, $X\approx Y$ denotes that there exist constants $c_{1},\,c_{2}$ such that $c_{1}X\leq Y\leq c_{2}X$.

\begin{proposition}\label{pro:stable}
Let $1< p,q<\infty $. Assume that $\phi\in W(L^{1,1})$ satisfies
\begin{equation}\label{eq:dayuling}
\sum_{k\in\mathbb{Z}^{d+1}}\left | \widehat{\phi}\left ( \xi +2\pi k \right )  \right |^{2}> 0, \quad \xi \in \mathbb{R}^{d+1}.
\end{equation}
Then for any $c\in \ell^{p,q}$, the function $f=\sum_{k_{1}\in \Bbb Z}\sum_{k_{2}\in \Bbb Z^d}c(k_{1},k_{2})\phi(\cdot-k_{1},\cdot-k_{2})$ belongs to $L^{p,q}$, and
$$
\|c\|_{\ell^{p,q}}\approx\|f\|_{L^{p,q}}.
$$
\end{proposition}

In order to define shift-invariant spaces reasonably, we require the generator $\phi$ of $V_{p,q}(\phi)$ satisfies (\ref{eq:dayuling}).

\subsection{Multiply generated shift-invariant space}
\ \ \ \
Let $\mathfrak{B}$ be a Banach space. $(\mathfrak{B})^{(r)}$ denotes $r$ copies $\mathfrak{B}\times\cdots\times\mathfrak{B}$ of $\mathfrak{B}$. If $C=(c_{1},c_{2},\cdots,c_{r})^{T}\in (\mathfrak{B})^{(r)}$, then we define the norm of $C$ by
\[
\|C\|_{(\mathfrak{B})^{(r)}}=\left(\sum_{j=1}^{r}\|c_{j}\|^{2}_{\mathfrak{B}}\right)^{1/2}.
\]

For $\Phi=(\phi_{1},\phi_{2},\cdots,\phi_{r})^{T}\in W(L^{1,1})^{(r)}$, the multiply generated shift-invariant space in the mixed Legesgue spaces $L^{p,q}$ is defined by
\begin{eqnarray*}
&&V_{p,q}(\Phi)=\left\{\sum_{j=1}^{r}\sum_{k_{1}\in \Bbb Z}\sum_{k_{2}\in \Bbb Z^{d}}c_{j}(k_{1},k_{2})\phi_{j}(\cdot-k_{1},\cdot-k_{2}):\right.\\
&& \quad\quad\quad\quad\quad\quad\quad\quad\left.c_{j}=\left\{c_{j}(k_{1},k_{2}):k_{1}\in \Bbb Z,k_{2}\in\Bbb Z^{d}\right\}\in \ell^{p,q}\right\}.
\end{eqnarray*}
It is easy to see that the three sum pointwisely converges almost everywhere. In fact, for any $1\leq j\leq r$, $c_{j}=\left\{c_{j}(k_{1},k_{2}):k_{1}\in \Bbb Z,k_{2}\in\Bbb Z^{d}\right\}\in\ell^{p,q}$ derives $c_{j}\in \ell^{\infty}.$ This combines $\Phi=(\phi_{1},\phi_{2},\cdots,\phi_{r})^{T}\in W(L^{1,1})^{(r)}$ gets
\begin{eqnarray*}
\sum_{j=1}^{r}\sum_{k_{1}\in \Bbb Z}\sum_{k_{2}\in \Bbb Z^{d}}\left|c_{j}(k_{1},k_{2})\phi_{j}(x-k_{1},y-k_{2})\right|&\leq& \sum_{j=1}^{r}\|c_{j}\|_\infty\sum_{k_{1}\in \Bbb Z}\sum_{k_{2}\in \Bbb Z^{d}}|\phi_{j}(x-k_{1},y-k_{2})|\\
&\leq&\sum_{j=1}^{r}\|c_{j}\|_\infty\| \phi_{j} \|_{W(L^{1,1})}\\
&\leq&\left(\sum_{j=1}^{r}\|c_{j}\|_\infty^{2}\right)^{1/2}\left(\sum_{j=1}^{r}\|\phi_{j}\|^{2}_{W(L^{1,1})}\right)^{1/2}\\
&=&\left(\sum_{j=1}^{r}\|c_{j}\|_\infty^{2}\right)^{1/2}\|\Phi\|_{W(L^{1,1})^{(r)}}<\infty\,(a.e.).
 \end{eqnarray*}

The following theorem gives that the multiply generated shift-invariant space is well-defined in $L^{p,q}$.

\begin{theorem}\label{thm:stableup}
Assume that $1\leq p,q<\infty $ and $\Phi=(\phi_{1},\phi_{2},\cdots,\phi_{r})^{T}\in W(L^{1,1})^{(r)}$.
Then for any $C=(c_{1},c_{2},\cdots,c_{r})^{T}\in (\ell^{p,q})^{(r)}$, the function
\[
f=\sum_{j=1}^{r}\sum_{k_{1}\in \Bbb Z}\sum_{k_{2}\in \Bbb Z^d}c_{j}(k_{1},k_{2})\phi_{j}(\cdot-k_{1},\cdot-k_{2})
\]
 belongs to $L^{p,q}$ and
$$
\|f\|_{L^{p,q}}\leq\|C\|_{(\ell^{p,q})^{(r)}}\left \| \Phi \right \|_{W(L^{1,1})^{(r)}}.
$$
\end{theorem}

\begin{proof}
Since $\|\cdot\|_{L^{p,q}}$ is a norm, one has
\begin{eqnarray}\label{eq:sanjiaobufengshi}
\|f\|_{L^{p,q}}&=&\left\|\sum_{k_{1}\in \Bbb Z}\sum_{k_{2}\in \Bbb Z^d}c_{j}(k_{1},k_{2})\phi_{j}(\cdot-k_{1},\cdot-k_{2})\right\|_{L^{p,q}}\nonumber\\
&\leq&\sum_{j=1}^{r}\left\|\sum_{k_{1}\in \Bbb Z}\sum_{k_{2}\in \Bbb Z^d}c_{j}(k_{1},k_{2})\phi_{j}(\cdot-k_{1},\cdot-k_{2})\right\|_{L^{p,q}}.
\end{eqnarray}
By Proposition \ref{pro:stableup1}, for any $1\leq j\leq r$
\begin{equation}\label{eq:danshengcheng}
\left\|\sum_{k_{1}\in \Bbb Z}\sum_{k_{2}\in \Bbb Z^d}c_{j}(k_{1},k_{2})\phi_{j}(\cdot-k_{1},\cdot-k_{2})\right\|_{L^{p,q}}\leq\|c_{j}\|_{\ell^{p,q}}\left \| \phi_{j} \right \|_{W(L^{1,1})}.
\end{equation}
Therefore, combining (\ref{eq:sanjiaobufengshi}) and (\ref{eq:danshengcheng}), one has
\begin{eqnarray*}
\|f\|_{L^{p,q}}&\leq&\sum_{j=1}^{r}\|c_{j}\|_{\ell^{p,q}}\left \| \phi_{j} \right \|_{W(L^{1,1})}\\
&\leq& \left(\sum_{j=1}^{r}\|c_{j}\|^{2}_{\ell^{p,q}}\right)^{1/2}
\left(\sum_{j=1}^{r}\left \| \phi_{j} \right \|^{2}_{W(L^{1,1})}\right)^{1/2}\\
&\leq&\|C\|_{(\ell^{p,q})^{(r)}}\left \| \Phi \right \|_{W(L^{1,1})^{(r)}}.
\end{eqnarray*}
\end{proof}

In order to obtain the main result in this section, we need to introduce the following two propositions.

\begin{proposition}\cite[Theorem 3.3]{Jia1991}\label{pro:stable if and only if}
Assume that $\phi \in W\left ( L^{1,1} \right )$. Then $\phi$ satisfies
\[
\sum_{k\in\mathbb{Z}^{d+1}}\left | \widehat{\phi}\left ( \xi +2\pi k \right )  \right |^{2}> 0, \quad \xi \in \mathbb{R}^{d+1},
\]
if and only if
there exists a function $g=\sum_{k_{1}\in \Bbb Z}\sum_{k_{2}\in \Bbb Z^d}d(k_{1},k_{2})\phi(\cdot-k_{1},\cdot-k_{2})$ such that
\[
\left \langle \phi\left ( \cdot -\alpha  \right ) , g \right \rangle=\delta _{0,\alpha }.
\] 
Here $d=\{d(k_{1},k_{2}):k_{1}\in \Bbb Z,k_{2}\in \Bbb Z^d\}\in \ell^{1}$.
\end{proposition}

\begin{proposition}\cite[Lemma 3.3]{LiLiu}\label{lem:gW}
The function $g$ in Proposition \ref{pro:stable if and only if} belongs to $W(L^{1,1})$.
\end{proposition}

\begin{theorem}\label{thm:stable}
For $1< p,q<\infty $. Assume that $\Phi=(\phi_{1},\phi_{2},\cdots,\phi_{r})^{T}\in W(L^{1})^{(r)}$ satisfies
\[
A\leq[\widehat{\Phi},\widehat{\Phi}](\xi), \quad \xi \in \mathbb{R}^{d+1},
\]
where $A>0$ and $[\widehat{\Phi},\widehat{\Phi}](\xi)=\left(\sum_{k\in\bbZ^{d+1}}\widehat{\phi}_{j}(\xi+2k\pi)\overline{\widehat{\phi}_{j'}(\xi+2k\pi)}\right)_{1\leq j\leq r,1\leq j'\leq r}$.
Then for any $C=(c_{1},c_{2},\cdots,c_{r})^{T}\in (\ell^{p,q})^{(r)}$ and $f=\sum_{j=1}^{r}\sum_{k_{1}\in \Bbb Z}\sum_{k_{2}\in \Bbb Z^d}c_{j}(k_{1},k_{2})\phi_{j}(\cdot-k_{1},\cdot-k_{2})$, one has
$$
\|C\|_{(\ell^{p,q})^{(r)}}\approx\|f\|_{L^{p,q}}\approx \|f\|_{W(L^{p,q})}.
$$
\end{theorem}

\begin{proof}
We first prove $\|C\|_{(\ell^{p,q})^{(r)}}\approx\|f\|_{L^{p,q}}.$

By Theorem \ref{thm:stableup}, we have
\[
\|f\|_{L^{p,q}}\lesssim\|C\|_{(\ell^{p,q})^{(r)}}.
\]
Thus we only need to obtain $\|C\|_{(\ell^{p,q})^{(r)}}\lesssim\|f\|_{L^{p,q}}$. Since for any $\xi\in\mathbb{R}^{d+1}$, $
A\leq[\widehat{\Phi},\widehat{\Phi}](\xi)$, we have
\[
\sum_{k\in\mathbb{Z}^{d+1}}\left | \widehat{\phi}_{j}\left ( \xi +2\pi k \right )  \right |^{2}> 0, \quad \xi \in \mathbb{R}^{d+1}, 1\leq j\leq r.
\]
By Proposition \ref{pro:stable if and only if} and Proposition \ref{lem:gW}, for each $\phi_{j}\,(1\leq j\leq r)$, there is a function $g_{j}\in W(L^{1,1})$ such that
\[
\left \langle \phi_{j}\left ( \cdot -\alpha  \right ) , g_{j} \right \rangle=\delta _{0,\alpha }.
\]
Therefore, for any $1\leq j\leq r$, $k_{1}\in\bbZ$ and $k_{2}\in\bbZ^{d}$
$$c_{j}(k_{1},k_{2})=\int_{\bbR}\int_{\bbR^d}f(x,y)\overline{g_{j}\left ( x-k_{1},y-k_{2} \right )}dxdy.$$
 Putting $B=(b_{1},b_{2},\cdots,b_{r})^{T}\in\left(\ell^{p',q'}\right)^{(r)}$ with $\frac{1}{p}+\frac{1}{p'}=1$ and $\frac{1}{q}+\frac{1}{q'}=1$. Then we have
\begin{eqnarray*}
\left | \left \langle C,B \right \rangle \right |&=&\left | \sum_{j=1}^{r}\sum_{k_{1}\in\mathbb{Z},k_{2}\in\mathbb{Z}^d} c_{j}(k_{1},k_{2}) \overline{b_{j}(k_{1},k_{2})}  \right |\\
&=&\left |\sum_{j=1}^{r}\sum_{k_{1}\in\mathbb{Z},k_{2}\in\mathbb{Z}^d}\overline{b_{j}(k_{1},k_{2})}\int_{\bbR}\int_{\bbR^d}f(x,y)
\overline{g_{j}\left ( x-k_{1},y-k_{2} \right )}dxdy\right |\\
&=&\left |\int_{\bbR}\int_{\bbR^d}f(x,y)\sum_{j=1}^{r}\sum_{k_{1}\in\mathbb{Z},k_{2}\in\mathbb{Z}^d}
\overline{b_{j}(k_{1},k_{2})}\overline{g_{j}\left ( x-k_{1},y-k_{2} \right )}dxdy\right|.
\end{eqnarray*}
By Theorem \ref{thm:stableup}, we have \begin{eqnarray*}
\left | \left \langle C,B \right \rangle \right |&\leq &\left \| f \right \|_{L^{p,q}}\left \| \sum_{j=1}^{r}\sum_{k_{1}\in\mathbb{Z},k_{2}\in\mathbb{Z}^d}b_{j}(k_{1},k_{2})g_{j}\left ( x-k_{1},y-k_{2} \right ) \right \|_{L^{p',q'}}\\
&\leq &\left \| f  \right \|_{L^{p,q}}\left \| B \right \|_{\left(\ell^{p',q'}\right)^{(r)}}\left \| G \right \|_{W(L^{1,1})^{(r)}}.
\end{eqnarray*}
Here $G=(g_{1},g_{2},\cdots,g_{r})^{T}\in W(L^{1,1})^{(r)}$.
Therewith
\begin{eqnarray}\label{eq:3.2}
\left \| C  \right \|_{(\ell^{p,q})^{(r)}}\leq \left \| f \right \|_{L^{p,q}}\left \| G \right \|_{W(L^{1,1})^{(r)}},
\end{eqnarray}
namely $\left \| C  \right \|_{(\ell^{p,q})^{(r)}}\lesssim\|f\|_{L^{p,q}}.$
Hence, we get $\|C\|_{(\ell^{p,q})^{(r)}}\approx\|f\|_{L^{p,q}}$.

Next, we prove $\|f\|_{L^{p,q}}\approx \|f\|_{W(L^{p,q})}.$

From proof of \cite[Theorem 3.4]{LiLiu}, $\|f\|_{L^{p,q}}\lesssim\|f\|_{W(L^{p,q})}$ is a well known fact.

Conversely, by Proposition \ref{pro:stableup1}, $\|C\|_{(\ell^{p,q})^{(r)}}\approx\|f\|_{L^{p,q}}$ and  the triangle inequality of norm
\begin{eqnarray*}
\|f\|_{W(L^{p,q})}&=&\left\|\sum_{j=1}^{r}\sum_{k_{1}\in \Bbb Z}\sum_{k_{2}\in \Bbb Z^d}c_{j}(k_{1},k_{2})\phi_{j}(\cdot-k_{1},\cdot-k_{2})\right\|_{W(L^{p,q})}\\
&\leq&\sum_{j=1}^{r}\left\|\sum_{k_{1}\in \Bbb Z}\sum_{k_{2}\in \Bbb Z^d}c_{j}(k_{1},k_{2})\phi_{j}(\cdot-k_{1},\cdot-k_{2})\right\|_{W(L^{p,q})}\\
&\lesssim&\sum_{j=1}^{r}\left\|\sum_{k_{1}\in \Bbb Z}\sum_{k_{2}\in \Bbb Z^d}c_{j}(k_{1},k_{2})\phi_{j}(\cdot-k_{1},\cdot-k_{2})\right\|_{L^{p,q}}\\
&\leq&\sum_{j=1}^{r}\|c_{j}\|_{\ell^{p,q}}\|\phi_{j}\|_{W(L^{1,1})}\\
&\leq& \left(\sum_{j=1}^{r}\|c_{j}\|^{2}_{\ell^{p,q}}\right)^{1/2}
\left(\sum_{j=1}^{r}\left \| \phi_{j} \right \|^{2}_{W(L^{1,1})}\right)^{1/2}\\
&\leq&\|C\|_{(\ell^{p,q})^{(r)}}\left \| \Phi \right \|_{W(L^{1,1})^{(r)}}\\
&\lesssim&\|C\|_{(\ell^{p,q})^{(r)}}\lesssim\|f\|_{L^{p,q}}.
\end{eqnarray*}
Therefore, we have $\|f\|_{L^{p,q}}\approx \|f\|_{W(L^{p,q})}.$ This completes the proof.
\end{proof}

\section{Fast reconstruction algorithm}
\ \ \ \
In this section, we mainly discuss the reconstruction of the functions in principal shift-invariant spaces. The fast reconstruction algorithm is our main result. It allows to exactly reconstruct the signals $f$ in principal shift-invariant subspaces when the sampling set $X=\{(x_{j},y_{k}):k,j\in \mathbb{J}\}$ is sufficiently dense.

Before giving the main result of this paper, we give some definitions.

The following definition can separate sampling points.

\begin{definition}
A bounded uniform partition of unity $\{\beta_{j,k}\}_{j,k\in\mathbb{J}}$ associated to a strongly-separated
sampling set $X=\{(x_{j},y_{k}):j,k\in\mathbb{J}\}$ is a set of functions satisfying
\begin{enumerate}
  \item $0\leq\beta_{j,k}\leq1, \forall\,j,k\in\mathbb{J},$
  \item $\mathrm{supp}\beta_{j,k}\subset B_{\gamma}(x_{j},y_{k}),$
  \item $\sum_{j\in\mathbb{J}}\sum_{k\in\mathbb{J}}\beta_{j,k}=1$.
 \end{enumerate}
Here $B_{\gamma}(x_{j},y_{k})$ is the open ball with center $(x_{j},y_{k})$ and radius $\gamma$. 
\end{definition}

If $f\in W_{0}(L^{p,q})$, we define
\[
Q_{X}f=\sum_{j\in\mathbb{J}}\sum_{k\in\mathbb{J}}f(x_{j},y_{k})\beta_{j,k}
\]
for the quasi-interpolant of the sequence $c_{j,k}=f(x_{j},y_{k})$.

The following definition can describe the density of the sampling set $X$.

\begin{definition}
If a set $X=\{(x_{j},y_{k}):k,j\in \mathbb{J},x_{k}\in\mathbb{R},y_{j}\in\mathbb{R}^{d}\}$ satisfies
\[
\bbR^{d+1}=\cup_{j,k}B_{\gamma}(x_{j},y_{k})\quad\mbox{for every}\,\gamma>\gamma_{0},
\]
then we say that the set $X$ is $\gamma_{0}$-dense in $\bbR^{d+1}$.
Here $B_{\gamma}(x_{j},y_{k})$ is the open ball with center $(x_{j},y_{k})$ and radius $\gamma$, and $\mathbb{J}$ is a countable index set.
\end{definition}

The following is the main result of this paper, which gives a fast iterative algorithm to reconstruct $f\in V_{p,q}(\phi)$ from its samples  $\{f(x_{j},y_{k}):j,k\in\mathbb{J}\}.$

\begin{theorem}\label{th:suanfa}
Assume that $\phi\in W_{0}(L^{1,1})$ whose support is compact. Let $P$ is a bounded projection from $L^{p,q}$ onto $V_{p,q}(\phi)$. Then there is a density $\gamma>0\,(\gamma=\gamma(p,q,P))$ such that any $f\in V_{p,q}(\phi)$ can be reconstructed
from its samples $\{f(x_{j},y_{k}):(x_{j},y_{k})\in X\}$ on any $\gamma$-dense set $X=\{(x_{j},y_{k}):j,k\in\mathbb{J}\}$ by the following iterative algorithm:
\begin{eqnarray}\label{eq:iterative algorithm}
\left\{
\begin{array}{rl}f_{1}=&PQ_{X}f \\
 f_{n+1}=&PQ_{X}(f-f_{n})+f_{n}.\\
\end{array}\right.
\end{eqnarray}
The iterates $f_{n}$ converges to $f$  in $L^{p,q}$ norms  and  uniformly. 
Furthermore, the convergence is geometric, namely,
%\[
%\|f-f_{n}\|_{L^{p,q}}\leq M_{1}\|f-f_{n}\|_{W(L^{p,q})}\leq M_{2}\alpha^{n}
%\]
\[
\|f-f_{n}\|_{L^{p,q}}\leq M\alpha^{n}
\]
for some $\alpha=\alpha(\gamma)<1$ and $M<\infty.$
%$M_{1},M_{2}<\infty$.
\end{theorem}

Before proving Theorem \ref{th:suanfa}, we introduce some useful lemmas.

Let $f$ be a continuous function. We define the oscillation (or modulus of continuity) of $f$
by $\hbox{osc}_{\delta}(f)(x_{1},x_{2})=\sup_{|y_{1}|\leq\delta}\sup_{|y_{2}|\leq\delta|}|f(x_{1}+y_{1}, x_{2}+y_{2})-f(x_{1},x_{2})|$. 

\begin{lemma}\label{pro:in Wiener space}
Let $\phi\in W_{0}(L^{1,1})$ whose support is compact. Then there exists $\delta_{0}>0$ such that
 $\hbox{osc}_{\delta}(\phi)\in W_{0}(L^{1,1})$ for any $\delta<\delta_{0}$.
\end{lemma}

\begin{proof}
Since $\phi$ is continuous and that its support is compact, given $\epsilon_{0}=1$, there exists  $\delta_{0}>0$ such that
\begin{eqnarray*}
|\phi(x_{1}+k_{1}+y_{1},x_{2}+k_{2})|&\leq& |\phi(x_{1}+k_{1}+y_{1},x_{2}+k_{2})-\phi(x_{1}+k_{1},x_{2}+k_{2})|\\
&&\quad+|\phi(x_{1}+k_{1},x_{2}+k_{2})|\\
&<&1+|\phi(x_{1}+k_{1},x_{2}+k_{2})|.
\end{eqnarray*}
for all $ x_{1}\in[0,1],x_{2}\in[0,1]^{d},\,k_{1}\in\bbZ,\,k_{2}\in\bbZ^{d}$, $\delta<\delta_{0}$ and $|y_{1}|<\delta$.
Note that $\phi$ is compactly supported again, one has that there exists $D_{0}>0$ such that
\begin{eqnarray*}
&&\sum_{k_{1}\in \bbZ}\sup_{x_{1}\in [0,1]}\sum_{k_{2}\in \bbZ^{d}}\sup_{x_{2}\in [0,1]^{d}}\sup_{|y_{1}|\leq\delta}|\phi(x_{1}+k_{1}+y_{1},x_{2}+k_{2})|\\
&&\quad <\sum_{k_{1}\in \bbZ}\sup_{x_{1}\in [0,1]}\sum_{k_{2}\in \bbZ^{d}}\sup_{x_{2}\in [0,1]^{d}}(1+|\phi(x_{1}+k_{1},x_{2}+k_{2})|)\\
&&\quad \leq D_{0} +\sum_{k_{1}\in \bbZ}\sup_{x_{1}\in [0,1]}\sum_{k_{2}\in \bbZ^{d}}\sup_{x_{2}\in [0,1]^{d}}|\phi(x_{1}+k_{1},x_{2}+k_{2})|.
\end{eqnarray*}
Therefore, for any $\delta<\delta_{0}$
\begin{eqnarray*}
&&\sum_{k_{1}\in \bbZ}\sup_{x_{1}\in [0,1]}\sum_{k_{2}\in \bbZ^{d}}\sup_{x_{2}\in [0,1]^{d}}|\hbox{osc}_{\delta}(\phi)(x_{1}+k_{1},x_{2}+k_{2})|\\
&&\quad=\sum_{k_{1}\in \bbZ}\sup_{x_{1}\in [0,1]}\sum_{k_{2}\in \bbZ^{d}}\sup_{x_{2}\in [0,1]^{d}}\sup_{|y_{1}|\leq\delta}\sup_{|y_{2}|\leq\delta}\\
&&\quad\quad\quad|\phi(x_{1}+k_{1}+y_{1},x_{2}+k_{2}+y_{2})-\phi(x_{1}+k_{1},x_{2}+k_{2})|\\
&&\quad\leq\sum_{k_{1}\in \bbZ}\sup_{x_{1}\in [0,1]}\sum_{k_{2}\in \bbZ^{d}}\sup_{x_{2}\in [0,1]^{d}}\sup_{|y_{2}|\leq\delta}\sup_{|y_{1}|\leq\delta}\\
&&\quad\quad\quad|\phi(x_{1}+k_{1}+y_{1},x_{2}+k_{2}+y_{2})-\phi(x_{1}+k_{1},x_{2}+k_{2})|\\
&&\quad\leq\sum_{k_{1}\in \bbZ}\sup_{x_{1}\in [0,1]}\sum_{k_{2}\in \bbZ^{d}}\sup_{x_{2}\in [-1,2]^{d}}\sup_{|y_{1}|\leq\delta}\\
&&\quad\quad\quad|\phi(x_{1}+k_{1}+y_{1},x_{2}+k_{2})|+|\phi(x_{1}+k_{1},x_{2}+k_{2})|\\
&&\quad\leq2^{2d}\sum_{k_{1}\in \bbZ}\sup_{x_{1}\in [0,1]}\sum_{k_{2}\in \bbZ^{d}}\sup_{x_{2}\in [0,1]^{d}}\sup_{|y_{1}|\leq\delta}\\
&&\quad\quad\quad|\phi(x_{1}+k_{1}+y_{1},x_{2}+k_{2})|+|\phi(x_{1}+k_{1},x_{2}+k_{2})|\\
&&\quad\leq2^{2d}\sum_{k_{1}\in \bbZ}\sup_{x_{1}\in [0,1]}\sum_{k_{2}\in \bbZ^{d}}\sup_{x_{2}\in [0,1]^{d}}\sup_{|y_{1}|\leq\delta}|\phi(x_{1}+k_{1}+y_{1},x_{2}+k_{2})|\\
&&\quad\quad+2^{2d}\sum_{k_{1}\in \bbZ}\sup_{x_{1}\in [0,1]}\sum_{k_{2}\in \bbZ^{d}}\sup_{x_{2}\in [0,1]^{d}}|\phi(x_{1}+k_{1},x_{2}+k_{2})|\\
&&\quad=2^{2d}(D_{0} +\sum_{k_{1}\in \bbZ}\sup_{x_{1}\in [0,1]}\sum_{k_{2}\in \bbZ^{d}}\sup_{x_{2}\in [0,1]^{d}}|\phi(x_{1}+k_{1},x_{2}+k_{2})|)\\
&&\quad\quad+2^{2d}\|\phi\|_{W(L^{1,1})}\\
&&\quad=2^{2d}(D_{0} +\|\phi\|_{W(L^{1,1})})+2^{2d}\|\phi\|_{W(L^{1,1})}<\infty.
\end{eqnarray*}
Thus, we conclude that there exists $\delta_{0}>0$ such that
 $\hbox{osc}_{\delta}(\phi)\in W_{0}(L^{1,1})$ for any $\delta<\delta_{0}$. 
\end{proof}

\begin{lemma}\label{lem:oscillation}
Let $\phi\in W_{0}(L^{1,1})$ whose support is compact. Then $\lim_{\delta\rightarrow0}\|\hbox{osc}_{\delta}(\phi)\|_{W(L^{1,1})}=0$.
\end{lemma}

\begin{proof}
Due to Lemma \ref{pro:in Wiener space}, one has that there exists $\delta_{0}>0$ such that
 $\hbox{osc}_{\delta}(\phi)\in W_{0}(L^{1,1})$ for any $\delta<\delta_{0}$.
For any $\delta<\delta_{0}$, given $\epsilon>0,$ there exists an integer $L_{0}>0$ such that
\begin{equation}\label{eq:1}
\sum_{k_{1}\in\bbZ}\sup_{x_{1}\in[0,1]}\sum_{|k_{2}|\geq L_{0}}\sup_{x_{2}\in[0,1]^{d}}|\hbox{osc}_{\delta}(\phi)(x_{1}+k_{1},x_{2}+k_{2})|<\frac{\epsilon}{3}
\end{equation}
and 
\begin{equation}\label{eq:2}
\sum_{|k_{1}|\geq L_{0}}\sup_{x_{1}\in[0,1]}\sum_{k_{2}\in\bbZ^{d}}\sup_{x_{2}\in[0,1]^{d}}|\hbox{osc}_{\delta}(\phi)(x_{1}+k_{1},x_{2}+k_{2})|<\frac{\epsilon}{3}.
\end{equation}
Futhermore,  since $\phi$ is continuous, there exists  $\delta_{1}>0$ such that
\[
|\hbox{osc}_{\delta}(\phi)(x_{1}+k_{1},x_{2}+k_{2})|\leq \frac{\epsilon}{2(2L_{0})^{d+1}}
\]
for all $ x_{1}\in[0,1],x_{2}\in[0,1]^{d},\,|k_{1}|<L_{0},\,|k_{2}|<L_{0}$ and $\delta<\delta_{1}$. Thus
\begin{eqnarray}\label{eq:3}
\sum_{|k_{1}|<L_{0}}\sup_{x_{1}\in[0,1]}\sum_{|k_{2}|<L_{0}}\sup_{x_{2}\in[0,1]^{d}}|\hbox{osc}_{\delta}(\phi)(x_{1}+k_{1},x_{2}+k_{2})|&<&2L_{0}(2L_{0})^{d}
\frac{\epsilon}{2(2L_{0})^{d+1}}\nonumber\\
&=&\frac{\epsilon}{3}.
\end{eqnarray}
Combining (\ref{eq:1}), (\ref{eq:2}) and (\ref{eq:3}), for $\delta<\min\{\delta_{0},\,\delta_{1}\}$, one obtains
\begin{eqnarray*}
\|\hbox{osc}_{\delta}(\phi)\|_{W(L^{1,1})}&=&\sum_{k_{1}\in\bbZ}\sup_{x_{1}\in[0,1]}\sum_{k_{2}\in\bbZ^{d}}\sup_{x_{2}\in[0,1]^{d}}|\hbox{osc}_{\delta}(\phi)(x_{1}+k_{1},x_{2}+k_{2})|\\
&=&\sum_{k_{1}\in\bbZ}\sup_{x_{1}\in[0,1]}\left(\sum_{|k_{2}|\geq L_{0}}\sup_{x_{2}\in[0,1]^{d}}|\hbox{osc}_{\delta}(\phi)(x_{1}+k_{1},x_{2}+k_{2})|\right.\\
&&\quad+\left.\sum_{|k_{2}|< L_{0}}\sup_{x_{2}\in[0,1]^{d}}|\hbox{osc}_{\delta}(\phi)(x_{1}+k_{1},x_{2}+k_{2})|\right)\\
&\leq&\sum_{k_{1}\in\bbZ}\left(\sup_{x_{1}\in[0,1]}\sum_{|k_{2}|\geq L_{0}}\sup_{x_{2}\in[0,1]^{d}}|\hbox{osc}_{\delta}(\phi)(x_{1}+k_{1},x_{2}+k_{2})|\right.\\
&&\quad+\left.\sup_{x_{1}\in[0,1]}\sum_{|k_{2}|< L_{0}}\sup_{x_{2}\in[0,1]^{d}}|\hbox{osc}_{\delta}(\phi)(x_{1}+k_{1},x_{2}+k_{2})|\right)\\
&=&\sum_{k_{1}\in\bbZ}\sup_{x_{1}\in[0,1]}\sum_{|k_{2}|\geq L_{0}}\sup_{x_{2}\in[0,1]^{d}}|\hbox{osc}_{\delta}(\phi)(x_{1}+k_{1},x_{2}+k_{2})|\\
&&\quad+\sum_{k_{1}\in\bbZ}\sup_{x_{1}\in[0,1]}\sum_{|k_{2}|< L_{0}}\sup_{x_{2}\in[0,1]^{d}}|\hbox{osc}_{\delta}(\phi)(x_{1}+k_{1},x_{2}+k_{2})|\\
&=&\sum_{k_{1}\in\bbZ}\sup_{x_{1}\in[0,1]}\sum_{|k_{2}|\geq L_{0}}\sup_{x_{2}\in[0,1]^{d}}|\hbox{osc}_{\delta}(\phi)(x_{1}+k_{1},x_{2}+k_{2})|\\
&&\quad+\sum_{|k_{1}|\geq L_{0}}\sup_{x_{1}\in[0,1]}\sum_{|k_{2}|< L_{0}}\sup_{x_{2}\in[0,1]^{d}}|\hbox{osc}_{\delta}(\phi)(x_{1}+k_{1},x_{2}+k_{2})|\\
&&\quad\quad+\sum_{|k_{1}|< L_{0}}\sup_{x_{1}\in[0,1]}\sum_{|k_{2}|< L_{0}}\sup_{x_{2}\in[0,1]^{d}}|\hbox{osc}_{\delta}(\phi)(x_{1}+k_{1},x_{2}+k_{2})|\\
&\leq&\sum_{k_{1}\in\bbZ}\sup_{x_{1}\in[0,1]}\sum_{|k_{2}|\geq L_{0}}\sup_{x_{2}\in[0,1]^{d}}|\hbox{osc}_{\delta}(\phi)(x_{1}+k_{1},x_{2}+k_{2})|\\
&&\quad+\sum_{|k_{1}|\geq L_{0}}\sup_{x_{1}\in[0,1]}\sum_{k_{2}\in\bbZ^{d}}\sup_{x_{2}\in[0,1]^{d}}|\hbox{osc}_{\delta}(\phi)(x_{1}+k_{1},x_{2}+k_{2})|\\
&&\quad\quad+\sum_{|k_{1}|< L_{0}}\sup_{x_{1}\in[0,1]}\sum_{|k_{2}|< L_{0}}\sup_{x_{2}\in[0,1]^{d}}|\hbox{osc}_{\delta}(\phi)(x_{1}+k_{1},x_{2}+k_{2})|\\
&<&\epsilon.
\end{eqnarray*}
Therefore $\lim_{\delta\rightarrow0}\|\hbox{osc}_{\delta}(\phi)\|_{W(L^{1,1})}=0$.
\end{proof}

In order to prove Theorem \ref{th:suanfa}, we also need the following two lemmas.

\begin{lemma}\label{lem:oscillation uniformly zero}
Assume that $\phi\in W_{0}(L^{1,1})$ whose support is compact and $f\in V_{p,q}(\phi)$. Then the oscillation (or modulus of continuity) $\hbox{osc}_{\delta}(f)$ belongs to $L^{p,q}$. Moreover for all $\epsilon>0$, there exists $\delta_{0}>0$ such that
$$\|\hbox{osc}_{\delta}(f)\|_{L^{p,q}}\leq \epsilon \|f\|_{L^{p,q}}$$
uniformly for all $0<\delta<\delta_{0}$ and $f\in V_{p,q}(\phi)$. 
\end{lemma}
%\|\hbox{osc}_{\delta}(f)\|_{W(L^{p,q})}\leq \epsilon \|f\|_{W(L^{p,q})}~\hbox{or}~

\begin{proof}
Let $f=\sum_{k_{1}\in \bbZ,\,k_{1}\in \bbZ^{d}}c(k_{1},k_{2})\phi(x_{1}-k_{1},x_{2}-k_{2})\in V^{p,q}(\phi)$. Then
\begin{eqnarray*}
\hbox{osc}_{\delta}(f)(x_{1},x_{2})&=&\sup_{|y_{1}|\leq\delta}\sup_{|y_{2}|\leq\delta}|f(x_{1}+y_{1},x_{2}+y_{2})-f(x_{1},x_{2})|\\
&\leq&\sup_{|y_{1}|\leq\delta}\sup_{|y_{2}|\leq\delta}\sum_{k_{1}\in \bbZ,\,k_{1}\in \bbZ^{d}}|c(k_{1},k_{2})|\times\\
&&|\phi(x_{1}+y_{1}-k_{1},x_{2}+y_{2}-k_{2})-\phi(x_{1}-k_{1},x_{2}-k_{2})|\\
&\leq&\sum_{k_{1}\in \bbZ,\,k_{1}\in \bbZ^{d}}|c(k_{1},k_{2})|\times\\
&&\sup_{|y_{1}|\leq\delta}\sup_{|y_{2}|\leq\delta}|\phi(x_{1}+y_{1}-k_{1},x_{2}+y_{2}-k_{2})-\phi(x_{1}-k_{1},x_{2}-k_{2})|\\
&\leq&\sum_{k_{1}\in \bbZ,\,k_{1}\in \bbZ^{d}}|c(k_{1},k_{2})|\hbox{osc}_{\delta}(\phi)(x_{1}-k_{1},x_{2}-k_{2})\\
&=&[|c|*_{sd}\hbox{osc}_{\delta}(\phi)](x_{1},x_{2}),
\end{eqnarray*}
where $|c|=\{|c(k_{1},k_{2})|:\,k_{1}\in\bbZ, \,k_{2}\in\bbZ^{d}\}$.
by Proposition \ref{pro:stableup}, Proposition \ref{pro:stable} and Lemma \ref{pro:in Wiener space}, there exists $M>0$ such that
\begin{eqnarray}\label{eq:4}
\|\hbox{osc}_{\delta}(f)\|_{L^{p,q}}
&\leq&\left\||c|*_{sd}\hbox{osc}_{\delta}(\phi)\right\|_{L^{p,q}}\nonumber\\
&\leq& \||c|\|_{\ell^{p,q}}\|\hbox{osc}_{\delta}(\phi)\|_{W(L^{1,1})}\nonumber\\
&\leq& \|c\|_{\ell^{p,q}}\|\hbox{osc}_{\delta}(\phi)\|_{W(L^{1,1})}\nonumber\\
&\leq& M\|f\|_{L^{p,q}}\|\hbox{osc}_{\delta}(\phi)\|_{W(L^{1,1})}.
\end{eqnarray}
Using Lemma \ref{lem:oscillation}, for any $\epsilon>0$, there exists $\delta_{0}>0$ such that
$$
\|\hbox{osc}_{\delta}(\phi)\|_{W(L^{1,1})}<\frac{\epsilon}{M},\quad~\forall\,\, 0<\delta\leq \delta_{0}.
$$
This combines
(\ref{eq:4}) obtains $\|\hbox{osc}_{\delta}(f)\|_{L^{p,q}}\leq \epsilon \|f\|_{L^{p,q}}$.
\end{proof}

\begin{lemma}\label{lem:oscillation operator}
Let $\phi\in W_{0}(L^{1,1})$ whose support is compact.  Assume that $P$ be any bounded projection from $L^{p,q}$ onto $V_{p,q}(\phi)$.
Then there is  $\gamma_{0}=\gamma_{0}(p,q,P)$ such that the operator $I-PQ_{X}$ is a
contraction on $V_{p,q}(\phi)$ for every strongly-separated $\gamma$-dense set $X$ with $\gamma\leq\gamma_{0}$.
\end{lemma}

\begin{proof}
Let $f\in V_{p,q}(\phi)$, we have
\begin{eqnarray*}
\|f-PQ_{X}f\|_{L^{p,q}}&=&\|P f-PQ_{X}f\|_{L^{p,q}}\\
&\leq&\|P\|_{op}\|f-Q_{X}f\|_{L^{p,q}}\\
&\leq&\|P\|_{op}\|\hbox{osc}_{\gamma}f\|_{L^{p,q}}\\
&\leq&\epsilon\|P\|_{op} \|f\|_{L^{p,q}}\,(\mbox{Lemma \ref{lem:oscillation uniformly zero}}),
\end{eqnarray*}
We can choose a $\gamma_0$ such that for any $\gamma<\gamma_0$,
$\epsilon\|P\|_{op}<1.$ Therefore, we get a contraction operator.
\end{proof}

Now, we are ready to prove the main result in this paper.

\begin{proof}[\textbf{Proof of Theorem \ref{th:suanfa}}]
Let $e_{n}=f-f_{n}$ be the error after $n$ iterations. Using (\ref{eq:iterative algorithm}),
%the sequence $e_{n}$ satisfies the recursion
\begin{eqnarray*}
e_{n+1}&=&f-f_{n+1}\\
&=&f-f_{n}-P Q_{X}(f-f_{n})\\
&=&(I-PQ_{X})e_{n}.
\end{eqnarray*}
By Lemma \ref{lem:oscillation operator}, we may choose a $\gamma$ so small that $\|I-PQ_{X}\|_{op}=\alpha<1$. Then one gets
\[
\|e_{n+1}\|_{L^{p,q}}\leq \alpha\|e_{n}\|_{L^{p,q}}\leq \alpha^{n}\|e_{0}\|_{L^{p,q}}.
\]
%and
%\[
%\|e_{n}\|_{W(L^{p,q})}\leq\alpha^{n}\|e_{0}\|_{W(L^{p,q})}.
%\]
Thus $\|e_{n}\|_{L^{p,q}}\rightarrow0$, when $n\rightarrow\infty$. This completes the proof.
\end{proof}

\section{Conclusion}
\ \ \ \
In this paper, we discuss reconstruction of functions in principal shift-invariant subspaces of mixed
Lebesgue spaces. We proposed a fast reconstruction algorithm which allows to exactly reconstruct the signals $f$ in the principal shift-invariant subspaces as long as the sampling set $X=\{(x_{j},y_{k}):k,j\in \mathbb{J}\}$ is sufficiently dense. Our results improve the result in \cite{LiLiu}. Studying nonuniform sampling problem in multiply generated vector shift-invariant subspaces of mixed Lebesgue spaces is the goal of future work.
\vspace{10pt}
\\
\textbf{Acknowledgements}

This work was supported partially by the
National Natural Science Foundation of China (11326094 and 11401435).
\\
\textbf{Competing interests}

The author declares that he has no competing interests.
\\
\textbf{Authors' contributions}

Qingyue Zhang provided the questions and gave the proof for the main results. He read and approved the manuscript.

\end{document}